%% file: tensorProductVerma.tex
\documentclass[12pt,reqno]{amsart}

\input{preamble/packages.tex}
\input{preamble/macros.tex}

\input{preamble/paperdata.tex}

\begin{document}

\input{preamble/abstract.tex}

\maketitle



\input{sections/intro.tex}

\input{sections/AK.tex}

\input{sections/gln.tex}
\input{sections/endomorphism.tex}
\input{sections/non-semisimple.tex}

\input{bibliography/bibliography.tex}

\end{document}

%% file: preamble/packages.tex
\usepackage{
	amsmath,
 	amsfonts,
  	amssymb,
 	amsthm,
        datetime,
	}

\usepackage{tikz-cd}
\usepackage{tikz}
\usepackage{ytableau}

\usepackage{rotating}

\usetikzlibrary{decorations.pathmorphing, decorations.pathreplacing}
 \usetikzlibrary{decorations.pathreplacing,backgrounds,decorations.markings}

\usepackage{color}

\usepackage{stmaryrd}
\usepackage[cal=boondoxo]{mathalfa}
\usepackage{mathtools}
\usepackage{bbm}

\usepackage{hyperref}
\usepackage[capitalise]{cleveref}

\usepackage{a4wide}

\usepackage{mathabx}


%% file: preamble/macros.tex
\theoremstyle{plain}
\newtheorem{thm}{Theorem}[section]
\newtheorem{cor}[thm]{Corollary}
\newtheorem{lem}[thm]{Lemma}
\newtheorem{prop}[thm]{Proposition}

\theoremstyle{definition}
\newtheorem{rem}[thm]{Remark}

\newtheorem{ex}[thm]{Example}

\theoremstyle{definition}
\newtheorem{defn}[thm]{Definition}

\newtheorem{thmABinner}{Theorem}
\newenvironment{thmAB}[1]{%
  \renewcommand\thethmABinner{#1}%
  \thmABinner
}{\endthmABinner}

\def\makeautorefname#1#2{\expandafter\def\csname#1autorefname\endcsname{#2}}
\makeautorefname{lem}{Lemma}%
\makeautorefname{prop}{Proposition}%
\makeautorefname{rem}{Remark}%
\makeautorefname{section}{Section}%

\newcommand{\cB}{\mathcal{B}}

\newcommand{\cP}{\mathcal{P}}

\newcommand{\cU}{\mathcal{U}}

\DeclareMathOperator{\id}{Id}

\DeclareMathOperator{\End}{End}

\DeclareMathOperator{\Res}{Res}

\definecolor{myblue}{rgb}{0,.5,1}
\definecolor{mygreen}{rgb}{.3,.75,.1}

\newcommand{\tikzdiagh}[2][]{\tikz[#1,very thick,baseline={([yshift=1ex+#2]current bounding box.center)}]}



%% file: preamble/paperdata.tex
\title[Schur--Weyl duality and Verma modules]{Schur--Weyl duality, Verma modules, and row quotients of Ariki--Koike algebras}

\author{Abel Lacabanne}
\address{Institut de Recherche en Math\'ematique et Physique\\
Universit\'e Catholique de Louvain\\ 
Chemin du Cyclotron 2\\ 
1348 Louvain-la-Neuve\\ 
Belgium}
\email{abel.lacabanne@uclouvain.be}
\author{Pedro Vaz}
\address{Institut de Recherche en Math\'ematique et Physique\\
Universit\'e Catholique de Louvain\\ 
Chemin du Cyclotron 2\\ 
1348 Louvain-la-Neuve\\ 
Belgium}
\email{pedro.vaz@uclouvain.be}


%% file: preamble/abstract.tex

\begin{abstract}
  We prove a Schur--Weyl duality between the quantum enveloping algebra of $\mathfrak{gl}_m$ and
  certain quotient algebras of Ariki--Koike algebras, which we describe explicitly.
  This duality involves several algebraically independent parameters and the module underlying it is a tensor product of a parabolic universal Verma module and a tensor power of the standard representation of $\mathfrak{gl}_m$.
We also give  a new presentation by generators and relations of the generalized blob algebras of Martin and Woodcock as well as an interpretation in terms of Schur--Weyl duality by showing they occur as a special case of our algebras.
\end{abstract}


%% file: sections/intro.tex

\section{Introduction}\label{sec:intro}

Schur--Weyl duality is a celebrated theorem connecting the finite-dimensional modules over the general linear and the symmetric groups.
It states that, over a field $\Bbbk$ that is algebraically closed, the actions of $GL_m(\Bbbk)$ and $\mathfrak{S}_n$ on $V=(\Bbbk^m)^{\otimes n}$ commute and form double centralizers.
Several variants of (quantum) Schur--Weyl duality are known, see for example~\cite{Ariki-Terasoma-Yamada,multiparameter_SW,BDEHJLNSS,CP-SW-affine,Jimbo,Sakamoto-Shoji} for such variants related to our paper.
One particular family of generalizations of interest for us uses a module akin to the one
appearing in Schur–Weyl duality, but with an infinite-dimensional module instead of $V$. 
For example, in~\cite{ILZ} it is established a Schur--Weyl duality between $\mathcal{U}_{q}(\mathfrak{sl}_2)$ and the blob algebra
of Martin and Saleur~\cite{Martin-Saleur} with the underlying module being a tensor product of a projective Verma module with several copies of the standard representation of $\mathcal{U}_{q}(\mathfrak{sl}_2)$. 
We should warn the reader that in~\cite{ILZ} the blob algebra was called the Temperley--Lieb algebra of type $B$ (see~\cite{LNV} for further explanations).

\subsection{In this paper}
We consider the tensor product of a parabolic universal Verma module with the $m$-folded tensor product of the standard representation for $\mathcal{U}_{q}(\mathfrak{gl}_m)$ to establish a Schur--Weyl duality with a quotient of Ariki--Koike algebras.
Ariki--Koike algebras were first considered by Cherednik in \cite{cherednik} as a cyclotomic quotient of the affine Hecke algebra of type $A$. These algebras were later rediscovered and studied by Ariki and Koike \cite{ariki-koike} from a representation theoretic point of view. Independently, Broué and Malle attached in \cite{broue-malle} a Hecke algebra to certain complex reflection groups, and Ariki--Koike algebras turn out to be the Hecke algebras associated to the complex reflection groups $G(d,1,n)$. 

\smallskip

Recall that the \emph{Ariki--Koike algebra} $\mathcal{H}(d,n)$ with parameters $q\in \Bbbk^*$ and $\underline{u}=(u_1,\ldots, u_d)\in \Bbbk^d$ is the $\Bbbk$-algebra with generators $T_0,T_1,\ldots T_{n-1}$, where $T_1,\ldots T_{n-1}$ generate a finite-dimensional Hecke algebra of type $A$ and $T_0$ satisfies
$T_0T_1T_0T_1 = T_1T_0T_1T_0$, $T_0T_i=T_iT_0$ for $i>1$, and $\prod_{i=1}^d (T_0-u_i) = 0$. 
We consider the semisimple case, where the simple modules $V_{\mu}$ of $\mathcal{H}(d,n)$ are indexed by $d$-partitions of $n$. 

\smallskip

Let $\underline{m}=(m_1,\ldots,m_d)$ be a $d$-tuple of positive integers and $\cP^n_{\underline{m}}$ be the set of all  $d$-partitions  $\mu=(\mu^{(1)},\ldots,\mu^{(d)})$ of $n$ such that $l(\mu^{(i)})\leq m_i$ for all $1\leq i \leq d$.

In this paper we introduce the \emph{row-quotient} algebra $\mathcal{H}_{\underline{m}}(d,n)$, that depends on $\underline{m}$ as the quotient of $\mathcal{H}(d,n)$ by the kernel of the surjection
  \[
    \mathcal{H}(d,n) \twoheadrightarrow \prod_{\mu\in\cP^n_{\underline{m}}}\End_{\Bbbk}\left(V_{\mu}\right) .  
\]

\smallskip

Let $M^{\mathfrak{p}}(\Lambda)$ be a parabolic Verma module and $V$ the standard representation for $\mathcal{U}_{q}(\mathfrak{gl}_m)$. 
In our conventions, $\mathfrak{p}$ is standard and has Levi factor  $\mathfrak{l}=\mathfrak{gl}_{m_1}\times\cdots\times\mathfrak{gl}_{m_d}$, with $m_i\geq 1$ and $m_1+m_2+\dotsm +m_d=m$  and $\Lambda$ depends on $d$ algebraically independent parameters $\lambda_1,\ldots,\lambda_d$ (see \cref{sec:parabolic} for more details). Thanks to the braided structure on the category of integrable modules over $\mathcal{U}_{q}(\mathfrak{gl}_m)$, we define a left action of $\mathcal{H}(d,n)$ on $ M^{\mathfrak{p}}(\Lambda)\otimes V^{\otimes n}$ in \cref{sec:endomorphism}.
Our main result is:
    \begin{thmAB}{A}[\cref{thm:main_result} and \cref{lem:eignevalues}]\leavevmode
      \begin{itemize}
      \item The action of $\mathcal{U}_q(\mathfrak{gl}_m)$ and of $\mathcal{H}(d,n)$ on $M^{\mathfrak{p}}(\Lambda)\otimes V^{\otimes n}$ commute with each other, which endow $M^{\mathfrak{p}}(\Lambda)\otimes V^{\otimes n}$ with a structure of $\mathcal{H}(d,n)\otimes\mathcal{U}_q(\mathfrak{gl}_m)$-module.
      \item The algebra morphism $\mathcal{H}(d,n)\rightarrow \End_{\mathcal{U}_{q}(\mathfrak{gl}_m)}( M^{\mathfrak{p}}(\Lambda)\otimes V^{\otimes n})$ is surjective and factors through an isomorphism
        \begin{equation}\label{eq:isoSW}
          \mathcal{H}_{\underline{m}}(d,n)\overset{\simeq}{\longrightarrow} \End_{\mathcal{U}_{q}(\mathfrak{gl}_m)}(M^{\mathfrak{p}}(\Lambda)\otimes V^{\otimes n}).
        \end{equation}
      \item There is an isomorphism of $\mathcal{H}(d,n)\otimes\mathcal{U}_q(\mathfrak{gl}_m)$-modules
        \[
          M^{\mathfrak{p}}(\Lambda)\otimes V^{\otimes n} \simeq \bigoplus_{\mu\in \mathcal{P}^{n}_{\underline{m}}}V_{\mu}\otimes M^{\mathfrak{p}}(\Lambda,\mu),
        \]
        where $M^{\mathfrak{p}}(\Lambda,\mu)$ is a simple module (see~\ref{sec:parabolic}).
      \end{itemize}
    \end{thmAB}
  
\smallskip

The isomorphism in Equation~\eqref{eq:isoSW} has several particular specializations (Corollaries~\ref{cor:HeckeA}-\ref{cor:genblob}), some of them recovering well-known algebras:
\begin{itemize}
\item If $\mathfrak{p}=\mathfrak{gl}_m$ and $m\geq n$, then $\End_{\mathcal{U}_{q}(\mathfrak{gl}_m)}( M^{\mathfrak{p}}(\Lambda)\otimes V^{\otimes n})$ is isomorphic to the Hecke algebra of type $A$.

\item If $\mathfrak{p}=\mathfrak{gl}_m$ and $m=2$, then $\End_{\mathcal{U}_{q}(\mathfrak{gl}_m)}( M^{\mathfrak{p}}(\Lambda)\otimes V^{\otimes n})$ is isomorphic to the Temperley--Lieb algebra of type $A$.

\item For $\mathfrak{p}$ such that $m\geq nd$ and $m_i\geq n$ for all $1\leq i \leq d$, then $\End_{\mathcal{U}_{q}(\mathfrak{gl}_m)}( M^{\mathfrak{p}}(\Lambda)\otimes V^{\otimes n})$ is isomorphic to the Ariki--Koike algebra $\mathcal{H}(d,n)$.

\item If $\mathfrak{p}$ is such that $d=2$ and $m_1,m_2\geq n$, then $\End_{\mathcal{U}_{q}(\mathfrak{gl}_m)}( M^{\mathfrak{p}}(\Lambda)\otimes V^{\otimes n})$ is isomorphic to the Hecke algebra of type $B$ with unequal and algebraically independent parameters (see \cite[Example 5.2.2, (c)]{geck-jacon}).

\item If the parabolic subalgebra $\mathfrak{p}$ coincides with the standard Borel subalgebra of $\cU_q(\mathfrak{gl}_m)$ then $\End_{\mathcal{U}_{q}(\mathfrak{gl}_m)}( M^{\mathfrak{p}}(\Lambda)\otimes V^{\otimes n})$ is isomorphic to Martin--Woodcock's~\cite{generalized_blob} generalized blob algebra $\cB(d,n)$.
This generalizes the case of $\cU_q(\mathfrak{sl}_2)$ covered in~\cite{ILZ}.
\end{itemize}

In the last case, this gives a new interpretation of the generalized blob algebras $\cB(d,n)$ in terms of Schur--Weyl duality.
We also give a new presentation of $\cB(d,n)$ as a quotient of Ariki--Koike algebras:
\begin{thmAB}{B}[\cref{thm:blob}]
  Suppose that $\mathcal{H}(d,n)$ is semisimple and that for every $i,j,k$ we have $(1+q^{-2})u_k \neq u_i + u_j$. The generalized blob algebra $\mathcal{B}(d,n)$ is isomorphic to the quotient of $\mathcal{H}(d,n)$ by the two-sided ideal generated by the element
  \[
   \tau = \prod_{1\leq i < j \leq d}\left[(T_1-q)\left(T_0-q\frac{u_i+u_j}{q+q^{-1}}\right)(T_1-q)\right] . 
  \]
\end{thmAB}


\subsection{Connection to other works}
The idea of writing this note originated when we started thinking of possible extensions of our work in~\cite{LNV} to more general Kac--Moody algebras and were not able to find the appropriate generalizations of~\cite{ILZ} in the literature. 
When we were finishing writing this note Peng Shan informed us about~\cite{RSVV}, whose results are far beyond the ambitions of this article. Nevertheless, we expect our results to be connected to~\cite[§8]{RSVV} using a braided equivalence of categories between a category of modules for the quantum group $\mathcal{U}_q(\mathfrak{gl}_m)$ and a category of modules over the affine Lie algebra $\widehat{\mathfrak{gl}}_m$, which is due to Kazhdan and Lusztig \cite{KL}. However, the explicit description of the endomorphism algebra of $M^{\mathfrak{p}}(\Lambda)\otimes V^{\otimes n}$, which was our first motivation towards categorification later on, does not seem to appear anywhere in~\cite{RSVV} except in the particular case of our~\cref{cor:AK}.

Another motivation for the results presented here resides in the potential applications to low-dimensional topology, as indicated in~\cite{RT}.
We find that it would be also interesting to investigate the use of several Verma modules in a tensor product as suggested in~\cite{DR}.


\subsection*{Acknowledgments}
%
%
%
We would like to thank Steen Ryom-Hansen for comments on an earlier version of this paper.
The authors would also like to thank the referee for his/her numerous, detailed, and helpful comments.
The authors were supported by the Fonds de la Recherche Scientifique - FNRS under Grant no.~MIS-F.4536.19.
%




%% file: sections/AK.tex

\section{Ariki--Koike algebras, row quotients and generalized blob algebras}

We recall the definition of Ariki--Koike algebras and define some quotients which will appear as endomorphism algebras of modules over a quantum group. As a particular case we recover the generalized blob algebras of Martin and Woodcock \cite{generalized_blob} and we obtain a presentation of these blob algebras that seems to be new. 

\subsection{Reminders on Ariki--Koike algebras}
\label{sec:def_AK}

Fix once and for all a field $\Bbbk$ and two positive integers $d$ and $n$ and choose 
elements $q \in \Bbbk^*$ and $u_1,\ldots,u_d \in \Bbbk$. We recall the definition of the Ariki--Koike algebra introduced in \cite{ariki-koike}, which we view as a quotient of the group algebra of the Artin--Tits braid group of type $B$.

\begin{defn}
  The \emph{Ariki--Koike algebra} $\mathcal{H}(d,n)$ with parameters $q\in\Bbbk^*$ and $\underline{u}=(u_1,\ldots, u_d)\in \Bbbk^d$ is the $\Bbbk$-algebra with generators $T_0,T_1,\ldots T_{n-1}$, the relation
  \[
    (T_i-q)(T_i+q^{-1})=0,
  \]
  the cyclotomic relation
  \[
    \prod_{i=1}^d (T_0-u_i) = 0,
  \]
  and the braid relations
  \begin{align*}
    T_iT_j &= T_iT_j\ \mathrm{if}\ \lvert i-j \rvert > 1, & T_iT_{i+1}T_i = T_{i+1}T_iT_{i+1}\ \mathrm{for}\ 1\leq i \leq n-2,
  \end{align*}
  \[
    T_0T_1T_0T_1 = T_1T_0T_1T_0.
  \]
\end{defn}

\begin{rem}
  We use different conventions than~\cite{ariki-koike}. In order to recover their definition, one should replace $q$ by $q^2$, $T_{0}$ by $a_1$, and $qT_{i-1}$ by $a_i$.  
\end{rem}


As in the type $A$ Hecke algebra, for any $w\in \mathfrak{S}_n$ we can define unambiguously $T_w$ by choosing any reduced expression of $w$.

It is shown in \cite{ariki-koike} that the algebra $\mathcal{H}(d,n)$ is of dimension $d^nn!$ and a basis is given in terms of Jucys--Murphy elements, which are recursively defined by $X_1 = T_0$ and $X_{i+1} = T_iX_iT_i$.

\begin{thm}[{\cite[Theorem 3.10, Theorem 3.20]{ariki-koike}}]
  \label{thm:basis_AK}
  A basis of $\mathcal{H}(d,n)$ is given by the set
  \[
    \left\{X_1^{r_1}\ldots X_{d}^{r_d}T_w\ \middle\vert\ 0\leq r_i < d, w\in \mathfrak{S}_n\right\}.
  \]
  Moreover, the center of $\mathcal{H}(d,n)$ is generated by the symmetric polynomials in $X_1,\ldots,X_d$.
\end{thm}

We end this section with a semisimplicity criterion due to Ariki \cite{ariki}, which in our conventions takes the following form.

\begin{thm}[{\cite[Main Theorem]{ariki}}]
  The algebra $\mathcal{H}(d,n)$ is semisimple if and only if
  \[
    \left(\prod_{\substack{-n < l < n\\ 1 \leq i < j \leq d}}(q^{2l}u_i-u_j)\right)\left(\prod_{1\leq i \leq n}(1+q^2+q^4+\ldots + q^{2(i-1)})\right) \neq 0.
  \]
\end{thm}


\subsection{Modules over Ariki--Koike algebras}

In this section, we suppose that the algebra $\mathcal{H}(d,n)$ is semisimple. In \cite{ariki-koike}, Ariki and Koike gave a construction of the simple $\mathcal{H}(d,n)$-modules, using the combinatorics of multipartitions.

\subsubsection{$d$-partitions and the Young lattice}

A partition $\mu$ of $n$ of length $l(\mu)=k$ is a non-increasing sequence $\mu_1 \geq \mu_2 \geq \cdots \geq \mu_k > 0$ of integers summing to $\lvert\mu\rvert=n$. A $d$-partition of $n$ is a $d$-tuple of partitions $\mu=(\mu^{(1)},\ldots,\mu^{(d)})$ such that $\sum_{i=1}^d\lvert\mu^{(i)}\rvert = n$. Given a $d$-partition $\mu$ its Young diagram is the set
\[
  [\mu]=\left\{(a,b,c)\in\mathbb{N}\times\mathbb{N}\times\{1,\ldots,d\}\ \middle\vert\  1 \leq a \leq l(\mu), 1 \leq b \leq \mu^{(c)}_a\right\},
\]
whose elements are called boxes. We usually represents a Young diagram as a $d$-tuple of sequences of left-aligned boxes, with $\mu_a^{(c)}$ boxes in the $a$-th row of the $c$-th component.

\begin{ex}
  The Young diagram of the $3$-partition $((2,1),\emptyset,(3))$ of $6$ is
  \[
    \ytableausetup{aligntableaux=center}
    \left(\ydiagram{2,1},\emptyset,\ydiagram{3}\right).
  \]
\end{ex}

A box $\gamma$ of $[\mu]$ is said to be \emph{removable} if $[\mu]\setminus \{\gamma\}$ is the Young diagram of a $d$-partition $\nu$, and in this case the box $\gamma$ is said to be \emph{addable} to $\nu$.

\begin{ex}
  The removable boxes of the $3$-partition $((2,1),\emptyset,(3))$ below are depicted with a cross
  \[
    \ytableausetup{mathmode,aligntableaux=center}
    \left(
      \begin{ytableau}
        \phantom{a} & \times \\
        \times
      \end{ytableau},
      \emptyset,
      \begin{ytableau}
        \phantom{a} & \phantom{a} & \times
      \end{ytableau}
      \right).
  \]
\end{ex}

With respect to the above the definitions, we will also use the evident notions of adding a box to a Young diagram or removing a box from a Young diagram.

We consider the Young lattice for $d$-partitions and some sublattices. It is a graph with vertices consisting of $d$-partitions of any integers, and there is an edge between two $d$-partitions if and only if one can be obtained from the other by adding or removing a box.

\begin{ex}
  The beginning of the Young lattice for $2$-partitions is the following:
  \[    
    \begin{tikzcd}
      & & \left(\emptyset,\emptyset\right)\ar[dl,no head]\ar[dr,no head] & & \\
      & \left(\ydiagram{1},\emptyset\right)\ar[dl,no head]\ar[d,no head]\ar[dr,no head] & & \left(\emptyset,\ydiagram{1}\right)\ar[dl,no head]\ar[d,no head]\ar[dr,no head] & \\
      \left(\ydiagram{1,1},\emptyset\right) & \left(\ydiagram{2},\emptyset\right) & \left(\ydiagram{1},\ydiagram{1}\right) & \left(\ydiagram{2},\emptyset\right) & \left(\ydiagram{1,1},\emptyset\right).      
    \end{tikzcd}
  \]
\end{ex}

If we fix $\underline{m}=(m_1,\ldots,m_d)\in \mathbb{N}^d$, we then define $\mathcal{P}^n_{\underline{m}}$ as the set of $d$-partitions $\mu$ such that $l(\mu^{(i)})\leq m_i$. We will also consider the corresponding sublattice of the Young lattice.

\begin{ex}
  For $m_1=1$ and $m_2=2$, the beginning of the Young lattice for $2$-partitions $\mu$ with $l(\mu^{(1)})\leq 1$ and $l(\mu^{(2)})\leq 2$ is the following:
  \[    
    \begin{tikzcd}
      & & \left(\emptyset,\emptyset\right)\ar[dl,no head]\ar[dr,no head] & & \\
      & \left(\ydiagram{1},\emptyset\right)\ar[d,no head]\ar[dr,no head] & & \left(\emptyset,\ydiagram{1}\right)\ar[dl,no head]\ar[d,no head]\ar[dr,no head] & \\
      & \left(\ydiagram{2},\emptyset\right) & \left(\ydiagram{1},\ydiagram{1}\right) & \left(\ydiagram{2},\emptyset\right) & \left(\ydiagram{1,1},\emptyset\right).     
    \end{tikzcd}
  \]
\end{ex}

We end this subsection with the notion of a standard tableau of shape $\mu$ where $\mu$ is a $d$-partition of $n$. Such a standard tableau is a bijection $\mathfrak{t}\colon [\mu]\rightarrow \{1,\ldots,n\}$ such that for all boxes $\gamma=(a,b,c)$ and $\gamma'=(a',b',c)$ we have $\mathfrak{t}(\gamma) < \mathfrak{t}(\gamma')$ if $a=a'$ and $b<b'$ or $a<a'$ and $b=b'$. Giving a standard tableau of shape $\mu$ is equivalent to giving a path in the Young lattice from the empty $d$-partition to the $d$-partition $\mu$.

\begin{ex}
  The standard tableau
  \[
    \left(
      \begin{ytableau}
        1 \\
        4
      \end{ytableau},
      \emptyset,
      \begin{ytableau}
        2 & 3 
      \end{ytableau}
      \right)
    \]
    of shape $((1,1),\emptyset,(2))$ correspond to the path
    \[
      \begin{tikzcd}
        \ytableausetup{aligntableaux=center,smalltableaux}
        \left(\emptyset,\emptyset,\emptyset\right)\ar[r] & \left(\ydiagram{1},\emptyset,\emptyset\right)\ar[r] & \left(\ydiagram{1},\emptyset,\ydiagram{1}\right)\ar[r,] & \left(\ydiagram{1},\emptyset,\ydiagram{2}\right)\ar[r] & \left(\ydiagram{1,1},\emptyset,\ydiagram{2}\right).
      \end{tikzcd}
    \]
\end{ex}

\subsubsection{Constructing the simple modules}

We present the construction of simple modules of the Ariki--Koike algebra following \cite[Section 3]{ariki-koike}. This construction is similar to the classical construction of simple modules of the symmetric group, the Hecke algebra of type $A$ or of the complex reflection group $G(d,1,n)$. This construction describes explicitly the action of the Ariki--Koike algebra on a vector space. For $\mu  = (\mu ^{(1)},\ldots,\mu ^{(d)})$ a $d$-mul\-ti\-par\-ti\-tion of $n$, we set
\[
  V_\mu  = \bigoplus_{\mathfrak{t}}\Bbbk v_{\mathfrak{t}},
\]
where the sum is over all the standard tableaux of shape $\mu$. Ariki and Koike gave an explicit action of the generators on the basis of $V_\mu$ given by the standard tableaux. The action of $T_0$ is diagonal with respect to this basis:
\[
  T_0 v_{\mathfrak{t}} = u_c v_{\mathfrak{t}},
\]
where $c$ is such that $\mathfrak{t}(1,1,c)=1$. The action of $T_i$ is more involved and depends on the relative positions of the numbers $i$ and $i+1$ in the tableau $\mathfrak{t}$:
\begin{enumerate}
\item if $i$ and $i+1$ are in the same row of the standard tableau $\mathfrak{t}$, then $T_i v_{\mathfrak{t}} = q v_{\mathfrak{t}}$,
\item if $i$ and $i+1$ are in the same column of the standard tableau $\mathfrak{t}$, then $T_i v_{\mathfrak{t}} = -q^{-1} v_{\mathfrak{t}}$,
\item if $i$ and $i+1$ neither appear in the same row nor the same column of the standard tableau $\mathfrak{t}$, then $T_i$ will act on the two dimensional subspace generated by $v_{\mathfrak{t}}$ and $v_{\mathfrak{s}}$, where $\mathfrak{s}$ is the standard tableau obtained from $\mathfrak{t}$ by permuting the entries $i$ and $i+1$. The explicit matrix is given in \cite{ariki-koike} and we will not need it.
\end{enumerate}

\begin{prop}[{\cite[Theorem 3.7]{ariki-koike}}]
  \label{prop:AK_rep}
  If $\mu$ is any $d$-multipartition of $n$, the space $V_\mu$ is a well-defined $\mathcal{H}(d,n)$-module and it is absolutely simple. A set of isomorphism classes of simple $\mathcal{H}(d,n)$-modules is moreover given by $\{V_\mu\}_{\mu}$, for $\mu$ running over the set of $d$-partitions of $n$.
\end{prop}

The action of the Jucys--Murphy elements is also diagonal in the basis of standard tableaux:
\begin{equation}
  \label{eq:JM_eigen}
  X_i v_{\mathfrak{t}} = u_cq^{2(b-a)}v_{\mathfrak{t}},
\end{equation}
where $\mathfrak{t}(a,b,c)=i$. A useful consequence of \cref{prop:AK_rep} is the following: if $V$ is a simple $\mathcal{H}(d,n)$-module and $v\in V$ is a common eigenvector for $X_1,\ldots,X_d$ with eigenvalues as in \eqref{eq:JM_eigen} for some standard tableau $\mathfrak{t}$ of shape $\mu$, then $V$ is isomorphic to $V_{\mu}$.

From the explicit description of the modules $V_{\mu}$, using the standard inclusion $\mathcal{H}(n,d)\hookrightarrow \mathcal{H}(n+1,d)$, it is easy to see that for any $d$-partition of $n+1$ we have
\[
  \Res_{\mathcal{H}(n,d)}^{\mathcal{H}(n+1,d)}(V_{\mu}) \simeq \bigoplus_{\nu} V_{\nu},
\]
where the sum is over all $d$-partition $\nu$ of $n$ whose Young diagram is obtained by deleting one removable box from the Young diagram of $\mu$. The branching rule of the inclusions $\mathcal{H}(1,d)\subset\mathcal{H}(2,d)\subset\cdots\subset\mathcal{H}(n,d)$ is therefore governed by the Young lattice of $d$-partitions.

\subsection{Row quotients of $\mathcal{H}(d,n)$ and generalized blob algebras}

We now define the row quotients of $\mathcal{H}(d,n)$ which will appear later as endomorphism algebras of a tensor product of modules for $\mathcal{U}_q(\mathfrak{gl}_m)$.

\begin{defn}
  Let $\underline{m}=(m_1,\ldots,m_d)\in\mathbb{N}^d$ and recall that the algebra $\mathcal{H}(d,n)$ is assumed to be semisimple, which implies that $\mathcal{H}(d,n) \simeq \prod_{\mu}\End_{\Bbbk}\left(V_{\mu}\right)$, the product being over all $d$-partitions of $n$. Recall also that $\mathcal{P}^n_{\underline{m}}$ is the set of $d$-partitions of $n$ with $i$-th component of length at most $m_i$.
  
  The \emph{$\underline{m}$-row quotient of} $\mathcal{H}(d,n)$, denoted $\mathcal{H}_{\underline{m}}(d,n)$, is the quotient of $\mathcal{H}(d,n)$ by the kernel of the surjection
  \[
    \mathcal{H}(d,n) \twoheadrightarrow \prod_{\mu\in \mathcal{P}^n_{\underline{m}}}\End_{\Bbbk}\left(V_{\mu}\right).
\]
\end{defn}

\begin{rem}
  \label{rem:asymptotic_AK}
  If $m_i\geq n$ for all $1\leq i \leq d$ then $\mathcal{H}_{\underline{m}}(d,n)\simeq \mathcal{H}(d,n)$.
\end{rem}

Similar to the case of $\mathcal{H}(d,n)$, we have inclusions $\mathcal{H}_{\underline{m}}(1,d)\subset\mathcal{H}_{\underline{m}}(2,d)\subset\cdots\subset\mathcal{H}_{\underline{m}}(n,d)$ and the branching rule is governed by the corresponding truncation of the Young lattice of $d$-partitions.

\medskip

\subsubsection{Generalized blob algebras}

In the particular case where $m_i=1$ for all $1 \leq i \leq d$, we recover the definition of the generalized blob algebras \cite[Equation (14)]{generalized_blob}, which we denote by $\mathcal{B}(d,n)$. Under a mild hypothesis on the parameters, we give a presentation of $\mathcal{B}(d,n)$.


We consider the following element of $\mathcal{H}(d,n)$:
\[
  \tau = \prod_{1\leq i < j \leq d}\left[(T_1-q)\left(T_0-q\frac{u_i+u_j}{q+q^{-1}}\right)(T_1-q)\right].
\]
This element may look cumbersome, but can be better understood thanks to the following lemma:

\begin{lem}
  \label{lem:relJM}
  The two-sided ideal of $\mathcal{H}(d,n)$ generated by $\tau$ is equal to the two-sided ideal generated by
  \[
      (T_1-q)\prod_{1\leq i < j \leq d}\left(X_1+X_2-(u_i+u_j)\right).
  \]
\end{lem}

\begin{proof}
  A simple computation in $\mathcal{H}(d,n)$ shows that
  \[
    (T_1-q)\left(T_0-q\frac{u_i+u_j}{q+q^{-1}}\right)(T_1-q) = q\left(X_1+X_2-(u_i+u_j)\right)(T_1-q).
  \]

  We therefore conclude using the fact that $(T_1-q)^2 = -(q+q^{-1})(T_1-q)$ and that $T_1$ commutes with $X_1+X_2$.
\end{proof}

We now investigate which $\mathcal{H}(d,n)$-modules $V_\mu$ factor through the quotient by the two-sided ideal generated by $\tau$.

\begin{prop}
  \label{prop:action_tau}
 The element $\tau$ acts by zero on $V_\mu$ if and only if $l(\mu^{(k)})\leq 1$ for every $k$ such that $(1+q^{-2})u_k \neq u_i + u_j$ for all $i,j$.
\end{prop}

\begin{proof}
  Suppose that $\mu$ and $k$ are such that $l(\mu^{(k)})\geq 2$ with $(1+q^{-2})u_k \neq u_i + u_j$ for all $i,j$. Then there exist a tableau $\mathfrak{t}$ of shape $\mu$ such that $1$ and $2$ are in the first two columns of the $k$-th component of the Young diagram of $\mu$. By definition of $V_\mu$, the generator $T_1$ acts on $v_{\mathfrak{t}}$ by multiplication by $-q^{-1}$. The Jucys--Murphy element $X_1$ acts on $v_{\mathfrak{t}}$ by multiplication by $u_k$ whereas the Jucys--Murphy element $X_2$ acts on $v_{\mathfrak{t}}$ by multiplication by $q^{-2}u_k$. Therefore, thanks to \cref{lem:relJM}, $\tau$ does not act by zero on $V_\mu$.

  It remains to check that $\tau$ acts by zero on $V_\mu$ with $l(\mu^{(k)})\leq 1$ whenever $(1+q^{-2})u_k \neq u_i + u_j$ for all $i,j$. Let $\mathfrak{t}$ be a standard tableau of shape $\mu$. If $1$ and $2$ are in the same component of the tableau $\mathfrak{t}$, then either $1$ and $2$ are in the same row and $T_1$ acts on $v_{\mathfrak{t}}$ by multiplication by $q$, either $1$ and $2$ are in the same column and $X_1+X_2$ acts on $\mathfrak{t}$ by multiplication by $(1+q^{-2})u_k$. The second case is possible only if there exists $i,j$ such that $(1+q^{-2})u_k = u_i+u_j$ and then $\tau$ acts by zero. If $1$ and $2$ are in two different Young diagrams and $X_1+X_2$ acts on $\mathfrak{t}$ by $u_k+u_l$, where $k$ (resp. $l$) is such that $\mathfrak{t}(1,1,k)=1$ (resp $\mathfrak{t}(1,1,l)=2$). In both cases, $\tau$ acts by zero. 
\end{proof} 

\begin{thm}
  \label{thm:blob}
  Suppose that $\mathcal{H}(d,n)$ is semisimple and that for every $i,j,k$ we have $(1+q^{-2})u_k \neq u_i + u_j$. The generalized blob algebra $\mathcal{B}(d,n)$ is isomorphic to the quotient of $\mathcal{H}(d,n)$ by the two-sided ideal generated by $\tau$.
  
\end{thm}

\begin{proof}
  Recall that we supposed that $m_1=\ldots=m_d=1$. Thanks to \cref{prop:action_tau}, the element $\tau$ is in the kernel of the surjection
  \[
    \mathcal{H}(d,n) \twoheadrightarrow \prod_{\mu\in \mathcal{P}^n_{\underline{m}}}\End_{\Bbbk}\left(V_{\mu}\right).
  \]
  Therefore, we have a surjection $\mathcal{H}(d,n)/\mathcal{H}(d,n) \tau \mathcal{H}(d,n) \twoheadrightarrow \mathcal{B}(d,n)$. Once again, thanks to \cref{prop:action_tau}, the simple modules of $\mathcal{H}(d,n)/\mathcal{H}(d,n) \tau \mathcal{H}(d,n)$ are exactly the $V_\mu$ with $\mu\in \mathcal{P}^n_{\underline{m}}$ which shows that the above surjection is an isomorphism.
\end{proof}


%% file: sections/gln.tex

\section{Quantum $\mathfrak{gl}_{m}$, parabolic Verma modules and tensor products}\label{sec:gln}

We recall the definition of the quantum enveloping algebra of $\mathfrak{gl}_m$,
and we also recall some basic properties of its modules, \emph{e.g.} concerning parabolic Verma modules. 

\subsection{The quantum enveloping algebra of $\mathfrak{gl}_{m}$}

Let $q$ be an indeterminate. The following definition of $\mathcal{U}_q(\mathfrak{gl}_{m})$ is over the field $\mathbb{Q}(q)$, but, via scalar extension, we will also consider it over a field containing $\mathbb{Q}(q)$ without further notice.

\begin{defn}
  The quantum enveloping algebra $\mathcal{U}_q(\mathfrak{gl}_{m})$ is the $\mathbb{Q}(q)$-algebra with generators $L_i^{\pm 1}, E_j$ and $F_j$, for $1\leq i \leq m$ and $1\leq j \leq m-1$ with the following relations:
  \begin{align*}
    L_i^{\pm 1}L_i^{\mp 1} &= 1, & L_iL_j &= L_jL_i,\\
    L_iE_j &= q^{\delta_{i,j}-\delta_{i,j+1}}E_jL_i, & L_iF_j &= q^{-\delta_{i,j}+\delta_{i,j+1}}F_jL_i,
  \end{align*}
  \[
    [E_i,F_j] = \delta_{i,j}\frac{L_iL_{i+1}^{-1}-L_i^{-1}L_{i+1}}{q-q^{-1}},
  \]
  and the quantum Serre relations
  \begin{align*}
    E_iE_j &= E_jE_i\ \mathrm{if}\ \lvert i-j \rvert > 1, & E_i^2E_{i\pm 1} - (q+q^{-1})E_iE_{i\pm 1}E_i + E_{i\pm 1}E_i^2 &= 0,\\
    F_iF_j &= F_jF_i\ \mathrm{if}\ \lvert i-j \rvert > 1, & F_i^2F_{i\pm 1} - (q+q^{-1})F_iF_{i\pm 1}F_i + F_{i\pm 1}F_i^2 &= 0.
  \end{align*}
\end{defn}

We endow it with a structure of a Hopf algebra, with comultiplication $\Delta$, counit $\varepsilon$ and antipode $S$ given on generators by the following:
\begin{align*}
  \Delta(L_i) &= L_i\otimes L_i, & \varepsilon(L_i) &= 1, & S(L_i) = L_i^{-1},\\
  \Delta(E_i) &= E_i\otimes 1 + L_iL_{i+1}^{-1} \otimes E_i, & \varepsilon(E_i) &= 0, & S(E_i) = -L_i^{-1}L_{i+1}E_i,\\
  \Delta(F_i) &= F_i\otimes L_i^{-1}L_{i+1} + 1 \otimes F_i, & \varepsilon(F_i) &= 0, & S(F_i) = -F_iL_iL_{i+1}^{-1}.\\  
\end{align*}

Set $\mathcal{U}_q(\mathfrak{gl}_{m})^0$ as the subalgebra generated by $(L_i)_{1\leq i \leq m}$, and $\mathcal{U}_q(\mathfrak{gl}_{m})^{\geq 0}$ as the subalgebra generated by $(L_i,E_j)_{\substack{1\leq i \leq m\\1\leq j \leq m-1}}$.

We denote by $P=\bigoplus_{i=1}^m\mathbb{Z}\varepsilon_i$ the weight lattice of $\mathfrak{gl}_m$ with $\mathbb{Z}$-basis given by the fundamental weights $(\varpi_i)_{1\leq i \leq m}$ where $\varpi_i=\varepsilon_1+\cdots+\varepsilon_i$. We denote by $Q$ the root lattice with $\mathbb{Z}$-basis given by the simple roots $(\alpha_i)_{1\leq i \leq d-1}$ where $\alpha_i=\varepsilon_i-\varepsilon_{i+1}$. Denote by $\Phi^+$ the set of positive roots, by $P^+$ the set of dominant weights for $\mathfrak{gl}_m$, that is $\mu=\sum_{i=1}^m\mu_i\varepsilon_i$ with $\mu_1\geq\mu_2\geq\cdots\geq\mu_m$. We also endow $P$ with the standard non-degenerate bilinear form: $\langle\varepsilon_i , \varepsilon_j\rangle=\delta_{i,j}$. The symmetric group $\mathfrak{S}_m$ acts on $P$ by permuting the coordinates and leaves the bilinear form $\langle\cdot,\cdot\rangle$ invariant. Finally, let $\rho$ be the half-sum of the positive roots.

We will often work with extensions $\mathbb{Z}[\beta_1,\ldots,\beta_k]\otimes_{\mathbb{Z}} P$, where the $\beta_i$'s are indeterminates and we also extend the bilinear form $\langle\cdot,\cdot\rangle$ to $\mathbb{Z}[\beta_1,\ldots,\beta_k]\otimes_{\mathbb{Z}} P$.

\subsection{Weights and parabolic Verma modules}
\label{sec:parabolic}

Suppose that our field $\Bbbk$ contains the field $\mathbb{Q}(q)$ and let $M$ be an $\mathcal{U}_q(\mathfrak{gl}_{m})$-module over the ground field $\Bbbk$. An element $v\in M$ is said to be a weight vector if $L_iv = \varphi(\varepsilon_i)v$, where $\varphi \colon P \rightarrow \Bbbk$ is the corresponding weight. The module $M$ is said to be a weight module if the action of the elements $L_1,\ldots,L_m$ is simultaneously diagonalizable. A highest weight module is a weight module $M$ such that $M=\mathcal{U}_q(\mathfrak{gl}_{m})v$, where $v$ is a weight vector such that $E_iv=0$ for $1\leq i \leq m-1$.

It is well-known that finite dimensional weight $\mathcal{U}_q(\mathfrak{gl}_m)$-modules of type $1$ are parameterized by the set $P^+$ of dominant weights.

In this paper, we will be interested in modules over the field $\mathbb{Q}(q,\lambda_1,\ldots,\lambda_k)$, where $\lambda_i=q^{\beta_i}$ is an indeterminate (recall that $q$ is formal and so $q^{\beta_i}$ is also formal). Moreover, we only consider type $1$ modules, where the weights are of the form
\[
  \varphi(\nu) = q^{\langle\mu,\nu\rangle},
\]
for some $\mu\in \mathbb{Z}[\beta_1,\ldots,\beta_k]\otimes_{\mathbb{Z}}P$ and for all $\nu\in P$.

\medskip

We now turn to parabolic Verma modules. Let $\mathfrak{p}$ be a standard parabolic subalgebra of $\mathfrak{gl}_m$ with Levi factor $\mathfrak{l}=\mathfrak{gl}_{m_1}\times\cdots\times\mathfrak{gl}_{m_d}$, where $m_i\geq 1$ and $\sum_{i=1}^dm_i=m$. Denote by $I$ the set $\left\{\tilde{m}_i\ \vert\ 1\leq i \leq d-1\right\}$, where $\tilde{m}_i=m_1+\ldots+m_i$, so that $\mathcal{U}_q(\mathfrak{l})$ is generated by $L_i,E_j$ and $F_j$ for $1 \leq i \leq m$ and $j\not\in I$ and $\mathcal{U}_q(\mathfrak{p})$ is generated by $L_i,E_j$ and $F_k$ for $1 \leq i \leq m$, $1\leq j \leq m-1$ and $k\not\in I$.
Denote by $P^+_i$ the set of dominant weights for $\mathfrak{gl}_{m_i}$. 
We identify the set $P^+_1\times\cdots\times P^+_d$ with the dominant weights $P^+_{\mathfrak{l}}$ of $\mathfrak{l}$ by the following map
\[
  (\mu^{(1)},\ldots,\mu^{(d)}) \rightarrow \sum_{i=1}^d\left(\sum_{j=1}^{m_i}\mu_{j}^{(i)}\varepsilon_{\tilde{m}_{i-1}+j}\right).
\]

For a dominant weight $\mu\in P_{\mathfrak{l}}^+$, we have an simple integrable finite dimensional $\mathcal{U}_q(\mathfrak{l})$-module $V^{\mathfrak{l}}(\Lambda,\mu)$ of highest weight
\[
  \Lambda_\mu=\sum_{i=1}^d\left(\sum_{j=1}^{m_i}(\beta_i+\mu_{j}^{(i)})\varepsilon_{\tilde{m}_{i-1}+j}\right).
\]
Indeed, one can check that ${\langle\Lambda_\mu,\alpha_i\rangle} \in \mathbb{N}$ for any $i\not\in I$. We turn this $\mathcal{U}_q(\mathfrak{l})$-module into a $\mathcal{U}_q(\mathfrak{p})$-module by setting $E_i V^{\mathfrak{l}}(\Lambda,\mu) =0$ for all $i\in I$. Then the parabolic Verma module $M^{\mathfrak{p}}(\Lambda,\mu)$ is
\[
  M^{\mathfrak{p}}(\Lambda,\mu) = \mathcal{U}_q(\mathfrak{gl}_{m})\otimes_{\mathcal{U}_q(\mathfrak{p})}V^{\mathfrak{l}}(\Lambda,\mu).
\]
It is a highest weight module of highest weight $\Lambda_{\mu}$. If $\mu=0$, then we will simply denote this module by $M^{\mathfrak{p}}(\Lambda)$ and its highest weight by $\Lambda$.

\begin{lem}
 For any $\mu \in P_{\mathfrak{l}}^+$, the parabolic Verma module $M^{\mathfrak{p}}(\Lambda,\mu)$ is simple.
\end{lem}

\begin{proof}
  Since for any $i\in I$ the scalar product $\langle \Lambda_\mu, \alpha_i\rangle$ is not an integer, as one easily checks, the claim follows.
\end{proof}

\begin{rem}
  If the parabolic subalgebra $\mathfrak{p}$ is the Borel subalgebra $\mathfrak{b}$ of upper triangular matrices, we have $\mathcal{U}_q(\mathfrak{p}) = \mathcal{U}_q(\mathfrak{gl}_m)^{\geq 0}$ and the parabolic Verma module $M^{\mathfrak{b}}(\Lambda)$ is the universal Verma module. The adjective universal means that any parabolic Verma module can be obtained from $M^{\mathfrak{b}}(\Lambda)$ by specialization of the parameters.
\end{rem}

In the rest of this article, all dominant weights $\mu\in P^+_{\mathfrak{l}}$ will satisfy $\mu_{m_i}^{(i)}\geq 0$ for all $1\leq i \leq d$, and it will be convenient to identify such a weight $\mu$ with the corresponding $d$-partition in $\mathcal{P}^n_{\underline{m}}$. 
We will use the same notation $\mu$ to denote the $d$-partition or the corresponding dominant weight.

We also denote by $V$ the standard representation of $\mathfrak{gl}_m$ of dimension $m$.  Explicitly, this is a highest weight module with highest weight $\varepsilon_1$, it has as a basis $v_1,\ldots,v_m$ and the action of $\mathcal{U}_q(\mathfrak{gl}_m)$ is given by
\[
  L_i\cdot v_j = q^{\delta_{i,j}}v_j,\quad E_i\cdot v_j = \delta_{i+1,j}v_{j-1}\quad\text{and}\quad F_i\cdot v_j = \delta_{i,j}v_{j+1}.
\]

\subsection{Tensor products and branching rule}

As $\mathcal{U}_q(gl_{m})$ is a Hopf algebra, its category of modules can be endowed with a tensor product. Explicitly, given $M$ and $N$ two modules over a ground ring $R$, the action of the generators on $M\otimes_R N$ is given using the comultiplication: for all $v\in M$ and $w\in N$, one have
\begin{multline*}
  L_i\cdot(v\otimes w) = L_i\cdot v \otimes L_i \cdot w,\quad E_i\cdot (v\otimes w) = E_i\cdot v\otimes w + L_iL_{i+1}^{-1}\cdot v \otimes E_i\cdot w\\\text{and}\quad F_i\cdot (v\otimes w) = F_i\cdot v\otimes L_i^{-1}L_{i+1}^{-1}\cdot w + v \otimes F_i\cdot w.
\end{multline*}

We will write $\otimes$ instead of $\otimes_R$ to simplify the notations. Since we will be interested in the endomorphism algebra of $M^{\mathfrak{p}}(\Lambda)\otimes V^{\otimes n}$, we start by understanding the decomposition of this module.

\begin{prop}
  \label{prop:tensor_MV}
  For any $\mu\in \mathcal{P}_{\mathfrak{l}}^n$, there is an isomorphism of $\mathcal{U}_q(\mathfrak{gl}_{m})$-modules
  \[
    M^{\mathfrak{p}}(\Lambda,\mu)\otimes V \simeq \bigoplus_{\nu\in\mathcal{P}_{\mathfrak{l}}^{n+1}}M^{\mathfrak{p}}(\Lambda,\nu),
  \]
  where the sum is over all $\nu\in\mathcal{P}_{\mathfrak{l}}^{n+1}$ whose Young diagram is obtained from the Young diagram of $\mu$ by adding one addable box.
\end{prop}

\begin{proof}
  We start by showing that $M^{\mathfrak{p}}(\Lambda,\mu)\otimes V$ has a filtration given by the $M^{\mathfrak{p}}(\Lambda,\nu)$ as in the statement. First, we have the following tensor identity:
  \[
    (\mathcal{U}_q(\mathfrak{gl}_m)\otimes_{\mathcal{U}_q(\mathfrak{p})} V^{\mathfrak{l}}(\Lambda,\mu)) \otimes V \simeq \mathcal{U}_q(\mathfrak{gl}_m)\otimes_{\mathcal{U}_q(\mathfrak{p})}(V^{\mathfrak{l}}(\Lambda,\mu) \otimes V).
  \]
  Noticing that $L\mapsto \mathcal{U}_q(\mathfrak{gl}_m)\otimes_{\mathcal{U}_q(\mathfrak{p})} L$ is an exact functor from the category of finite dimensional $\mathcal{U}_q(\mathfrak{p})$-modules to the category of $\mathcal{U}_q(\mathfrak{gl}_m)$-modules, it remains to show that
  \[
    V^{\mathfrak{l}}(\Lambda,\mu) \otimes V \simeq \bigoplus_{\nu\in\mathcal{P}_{\mathfrak{l}}^{n+1}}V^{\mathfrak{l}}(\Lambda,\nu),
  \]
  where the sum is over all $\nu\in\mathcal{P}_{\mathfrak{l}}^{n+1}$ whose Young diagram is obtained from the Young diagram of $\mu$ by adding one addable box. This follows from the usual branching rule for $\mathcal{U}_q(\mathfrak{gl}_{m_i})$-modules.

To show that the sum is direct, we use arguments from the infinite-dimensional representation theory of Lie algebras. We consider the usual category $\mathcal{O}$ for $\mathcal{U}_q(\mathfrak{gl}_m)$ \cite[Chapter 4]{mazorchuk}. We then show that each $M^{\mathfrak{p}}(\Lambda,\nu)$ lie in a different block of the category $\mathcal{O}$, which then implies that the sum is direct.

  First, as $M^{\mathfrak{p}}(\Lambda,\nu)$ is a quotient of the universal Verma module $M^{\mathfrak{b}}(\Lambda_\nu)$, these two modules share the same central character. Therefore $M^{\mathfrak{p}}(\Lambda,\nu)$ and $^{\mathfrak{p}}(\Lambda,\nu')$ are in the same block if and only if the central characters afforded by $M^{\mathfrak{b}}(\Lambda_\nu)$ and $M^{\mathfrak{b}}(\Lambda_{\nu'})$ are the same. But these central characters are equal if and only if $\Lambda_\nu$ and $\Lambda_{\nu'}$ are in the same orbit for the dot action of the symmetric group, which is the usual action of the symmetric group shifted by the sum of simple roots $\rho$.  

  We obtain that $M^{\mathfrak{p}}(\Lambda,\nu)$ and $M^{\mathfrak{p}}(\Lambda,\nu')$ are in the same block if and only if there exists $w\in\mathfrak{S}_m$ such that
  \[
    w\cdot \Lambda_\nu = \Lambda_{\nu'}.
  \]

  Now, suppose that $M^{\mathfrak{p}}(\Lambda,\nu)$ and $M^{\mathfrak{p}}(\Lambda,\nu')$ are in the same block. Since the dot action satisfies $w\cdot (\eta+\gamma) = w\cdot \eta + w(\gamma)$, we deduce that $w(\Lambda) = \Lambda$ so that $w$ lies in $\mathfrak{S}_{m_1}\times \cdots \times \mathfrak{S}_{m_d}$. Then, writing $w=(w_1,\ldots,w_d)$, we find that $w_i\cdot \nu^{(i)} = \nu'^{(i)}$ for every $1\leq i \leq d$. Since both $\nu^{(i)}$ and $\nu'^{(i)}$ are dominant weights, we deduce that $\nu^{(i)}=\nu'^{(i)}$ for every $1 \leq i \leq d$. Indeed, each orbit for the dot action contains a unique dominant weight.

  Hence if $\nu\neq \nu'$, the parabolic Verma modules $M^{\mathfrak{p}}(\Lambda,\nu)$ and $M^{\mathfrak{p}}(\Lambda,\nu')$ are in different blocks of the category $\mathcal{O}$.
\end{proof}

Using the previous proposition and induction, one shows the following corollary.

\begin{cor}
  \label{cor:dec_tensor}
  There is an isomorphism
  \[
    M^{\mathfrak{p}}(\Lambda)\otimes V^{\otimes n} \simeq \bigoplus_{\mu\in \mathcal{P}^{n}_{\mathfrak{l}}}M(\Lambda,\mu)^{n_\mu},
  \]
  where $n_\mu$ is the number of paths from the empty $d$-partition to $\mu$ in the Young lattice of $d$-multipartitions.
\end{cor}

\subsection{Braiding and an action of the Artin-Tits group of type $B$}

The quantized enveloping algebra (or rather a completion of the tensor product with itself) contains an element, called the quasi-$R$-matrix, which is a crucial tool in defining a braiding on a subcategory of the $\mathcal{U}_q(\mathfrak{gl}_{m})$-modules. Since there are several possible braidings, we make our choice explicit and refer to \cite[10.1.D]{chari-pressley} for more details. 

In a completion of $\mathcal{U}_q(\mathfrak{gl}_m)\otimes \mathcal{U}_q(\mathfrak{gl}_m)$, we define an element $\Theta$ by
\[
  \Theta=\prod_{\alpha\in\Phi^+} \left(\sum_{n=0}^{+\infty}q^{\frac{n(n-1)}{2}}\frac{(q-q^{-1})^n}{[n]!} E_\alpha^{n}\otimes F_\alpha^{n}\right),
\]
where $[n]!=\prod_{i=1}^n\frac{q^i-q^{-i}}{q-q^{-1}}$ and $E_\alpha,F_\alpha$ being the root vectors associated to a positive root $\alpha$. If $M$ and $N$ are two $\mathcal{U}_q(\mathfrak{gl}_{m})$ type $1$  weight modules over the ground ring $\mathbb{Q}(q,\lambda_1,\ldots,\lambda_{d-1})$ where $\mathcal{U}_q(\mathfrak{gl}_{m})^{>0}$ act locally nilpotently, $\Theta$ induces an isomorphism of vector spaces $\Theta_{M,N}\colon M\otimes N \rightarrow M \otimes N$. We then define a morphism of $\mathcal{U}_q(\mathfrak{gl}_m)$-modules
\[
  c_{M,N}\colon M\otimes N \rightarrow N \otimes M,
\]
by
\[
c_{M,N} = \tau \circ f \circ \Theta_{M,N}, \]
where $\tau$ is the flip $v\otimes w \mapsto w\otimes v$ and $f$ is the map $v\otimes w \mapsto q^{\langle\mu,\nu\rangle}v\otimes w$ if $v$ and $w$ are of respective weights $\mu$ and $\nu$. This endows the category of type $1$ weight modules on which $\mathcal{U}_q(\mathfrak{gl}_{m})^{>0}$ acts locally nilpotently with a braiding. In particular, we have the hexagon equation:
\[
  c_{L\otimes M,N} = (c_{L,N}\otimes\id_M)\circ(\id_L\otimes c_{M,N})
  \quad\text{and}\quad
  c_{L,M\otimes N} = (\id_{M}\otimes c_{L,N})\circ(c_{L,M}\otimes \id_{N}).
\]

Let $\mathcal{B}_n$ be the Artin-Tits braid group of type $B_n$. It has the following presentation in terms of generators and relations:
\[
  \mathcal{B}_n = \left\langle
    \tau_0,\tau_1,\ldots,\tau_{n-1}\middle\vert
    \begin{array}{ll}
      \tau_0\tau_1\tau_0\tau_1=\tau_1\tau_0\tau_1\tau_0,&\\
      \tau_i\tau_j=\tau_j\tau_i,&\text{if}\ \lvert i-j \rvert > 1,\\
      \tau_i\tau_{i+1}\tau_i = \tau_{i+1}\tau_i\tau_{i+1},& \text{for}\ 1 \leq i \leq n-2
  \end{array}
  \right\rangle.
\]

Using the braiding, we define the following endomorphisms of $M\otimes N^{\otimes n}$:
\begin{align*}
  R_0 &= (c_{N,M} \circ c_{M,N})\otimes\id_{N^{\otimes n-1}},\\
  R_i &= \id_{M\otimes N^{\otimes i-1}}\otimes c_{N,N} \otimes \id_{N^{\otimes n-i-1}}, \ \text{for} \ 1\leq i \leq n-1.
\end{align*}

Pictorially, one can represent these endomorphisms as
\[
  R_0=\tikzdiagh{-0.7ex}{
    \draw (1,-1) .. controls (1,-0.5) and (0, -0.5) .. (0,0);
    \draw[line width = 2mm,white] (0,-1) .. controls (0,-0.5) and (1, -0.5) .. (1,0);
    \draw (0,-1) .. controls (0,-0.5) and (1, -0.5) .. (1,0);
    \draw (0,-1) .. controls (0,-1.5) and (1, -1.5) .. (1,-2) node[below]{\tiny $N$};
    \draw[line width = 2mm,white] (0,-2) .. controls (0,-1.5) and (1, -1.5) .. (1,-1);
    \draw (0,-2) node[below]{\tiny $M$} .. controls (0,-1.5) and (1, -1.5) .. (1,-1);
    \draw (2,0) -- (2,-2) node[below]{\tiny $N$};
    \node at(2.5,-1) {\tiny$\dots$};
    \draw (3,0) -- (3,-2) node[below]{\tiny $N$};
  }
  \quad\text{and}\quad
  R_i=\tikzdiagh{0.8ex}{
    \draw (0,0) -- (0,-2) node[below]{\tiny $M$};
    \draw (1,0) -- (1,-2) node[below]{\tiny $N$};
    \node at(1.5,-1) {\tiny$\dots$};
    \draw (2,0) .. controls (2,-1) and (3,-1) .. (3,-2) node[below]{\tiny $N$};;
    \draw[line width = 2mm,white] (3,0) .. controls (3,-1) and (2,-1) .. (2,-2);
    \draw (3,0) .. controls (3,-1) and (2,-1) .. (2,-2) node[below]{\tiny $N$};;
    \node at(3.5,-1) {\tiny$\dots$};
    \draw (4,0) -- (4,-2) node[below]{\tiny $N$};
    \draw[decoration={brace,mirror,raise=-8pt},decorate]  (0.85,-2.75) -- node {$i$} (2.15,-2.75)
  }   
\]

\begin{prop} \label{prop:usefulprop}
  The assignment $\tau_i\mapsto R_i$ defines an action of $\mathcal{B}_n$ on the module $M\otimes N^{\otimes n}$ which commutes with the $\mathcal{U}_q(\mathfrak{gl}_{m})$ action.
\end{prop}

\begin{proof}
  The fact that $R_i$ is a $\mathcal{U}_q(\mathfrak{gl}_{m})$-morphism follows by definition of $R_i$. The fact that the defining relations of $\mathcal{B}_n$ are satisfied follows from the embedding of the braid group of type $B_n$ into the braid group of type $A_{n+1}$ \cite[Lemma 2.1]{ILZ}. 
\end{proof}

Finally, we end this section with a lemma due to Drinfeld \cite[Proposition 5.1 and Remark 4) below]{drinfeld_almost} computing the action of the double braiding on highest weight modules, which is related with the action of the ribbon element.

\begin{lem}
  \label{lem:scalar_double}
  Let $L, M$ and $N$ be highest weight modules of respective highest weight $\lambda,\mu$ and $\nu$ such that $L \subset M\otimes N$. Then the double braiding $c_{N,M}\circ c_{M,N}$ restricted to $N$ acts by multiplication by the scalar
  \[
    q^{\langle \lambda,\lambda+2\rho\rangle - \langle\mu,\mu+2\rho\rangle - \langle\nu,\nu+2\rho\rangle}.
  \]
\end{lem}



%% file: sections/endomorphism.tex

\section{The endomorphism algebra of $M^{\mathfrak{p}}(\Lambda)\otimes V^{\otimes n}$}
\label{sec:endomorphism}

The aim of this section is to prove the main result of this paper. We first explain why $M^{\mathfrak{p}}(\Lambda)\otimes V^{\otimes n}$ inherits an action of the Ariki--Koike algebra from the action of the braid group of type $B_n$. It is a classical result that the eigenvalues of $R_i$ are $q$ and $-q^{-1}$: the action of the braiding on $V\otimes V$ is
\[
  v_i\otimes v_j \mapsto
  \begin{cases}
    q v_j\otimes v_i & \ \text{if}\ i=j,\\
    v_j\otimes v_i & \ \text{if}\ i>j,\\
    v_j\otimes v_i +(q-q^{-1})v_i\otimes v_j & \ \text{if}\ i< j.
  \end{cases}
\]
Moreover, using \cref{lem:scalar_double}, we easily compute the eigenvalues of the endomorphism $R_0$ in order to show that the action of $\mathcal{B}_n$ factors through the Ariki--Koike algebra.

\begin{lem}
  \label{lem:eignevalues}
  The eigenvalues $u_1,\ldots,u_d$ of $R_0$ on $M^{\mathfrak{p}}(\Lambda)\otimes V$ are equal to
  \[
    u_i=(\lambda_iq^{-\tilde{m}_{i-1}})^2.
  \]
\end{lem}

\begin{proof}
  Let $\Lambda$ be the highest weight of $M^{\mathfrak{p}}(\Lambda)$. The decomposition of $M^{\mathfrak{p}}(\Lambda)\otimes V$ is given in \cref{prop:tensor_MV}:
  \[
    M^{\mathfrak{p}}(\Lambda)\otimes V \simeq \bigoplus_{i=1}^{d}M^{\mathfrak{p}}(\Lambda,\mu_i),
  \]
  where $\mu_i$ is the $d$-partition of $1$ whose only non-zero component is the $i$-th one and is equal to $(1)$. The highest weight of $M^{\mathfrak{p}}(\Lambda,\mu_i)$ being $\Lambda+\varepsilon_{\tilde{m}_{i-1}+1}$, the action of $R_0$ on $M^{\mathfrak{p}}(\Lambda,\mu_i)$ is given by
  \[
    q^{\langle\Lambda+\varepsilon_{\tilde{m}_{i-1}+1},\Lambda+\varepsilon_{\tilde{m}_{i-1}+1}+2\rho\rangle-\langle\Lambda,\Lambda+2\rho\rangle-\langle\varepsilon_1,\varepsilon_1+2\rho\rangle},
  \]
  and we check that
  \[
    \langle\Lambda+\varepsilon_{\tilde{m}_{i-1}+1},\Lambda+\varepsilon_{\tilde{m}_{i-1}+1}+2\rho\rangle-\langle\Lambda,\Lambda+2\rho\rangle-\langle\varepsilon_1,\varepsilon_1+2\rho\rangle = 2(\beta_{i}-\tilde{m}_{i-1}).
  \]
\end{proof}

By the definition of the Ariki–Koike algebra, Proposition~\ref{prop:usefulprop} and the previous lemma we thus get an action of the Ariki–Koike algebra for the parameters $u_i=(\lambda_iq^{-\tilde{m}_{i-1}})^2$ on $M^{\mathfrak{p}}(\Lambda)\otimes V^{\otimes n}$. Therefore, the assignment $T_i\mapsto R_i$ defines a morphism of algebras
\[
  \mathcal{H}(d,n) \rightarrow \End_{\mathcal{U}_q(\mathfrak{gl}_m)}(M^{\mathfrak{p}}(\Lambda)\otimes V^{\otimes n}).
\]

\begin{thm}
  \label{thm:main_result}\leavevmode
  \begin{itemize}
  \item The algebra morphism $\mathcal{H}(d,n)\rightarrow \End_{\mathcal{U}_{q}(\mathfrak{gl}_m)}( M^{\mathfrak{p}}(\Lambda)\otimes V^{\otimes n})$ is surjective and factors through an isomorphism
    \[
      \mathcal{H}_{\underline{m}}(d,n)\overset{\simeq}{\longrightarrow} \End_{\mathcal{U}_{q}(\mathfrak{gl}_m)}( M^{\mathfrak{p}}(\Lambda)\otimes V^{\otimes n}).
    \]
  \item There is an isomorphism of $\mathcal{H}(d,n)\otimes\mathcal{U}_q(\mathfrak{gl}_m)$-module
    \[
      M^{\mathfrak{p}}(\Lambda)\otimes V^{\otimes n} \simeq \bigoplus_{\mu\in \mathcal{P}^{n}_{\underline{m}}}V_{\mu}\otimes M^{\mathfrak{p}}(\Lambda,\mu).
    \]
  \end{itemize}
\end{thm}

\begin{proof}
 The first part of the theorem follows immediately from the second part and the definition of the row-quotient $\mathcal{H}_{\underline{m}}(d,n)$.
  
  Using \cref{cor:dec_tensor} and the fact that $\mathcal{H}(d,n)$ acts on $M^{\mathfrak{p}}(\Lambda)\otimes V^{\otimes n}$ by $\mathcal{U}_q(\mathfrak{gl}_m)$-linear endomorphisms, we see that
  \[
    M^{\mathfrak{p}}(\Lambda)\otimes V^{\otimes n} \simeq \bigoplus_{\mu\in \mathcal{P}^{n}_{\mathfrak{l}}}\tilde{V}_{\mu}\otimes M^{\mathfrak{p}}(\Lambda,\mu),
  \]
  for some $\mathcal{H}(d,n)$-modules $\tilde{V}_{\mu}$. Since the multiplicity of $M^{\mathfrak{p}}(\Lambda,\mu)$ in $M^{\mathfrak{p}}(\Lambda)\otimes V^{\otimes n}$ is given by the number of paths in the Young lattice from the empty $d$-partition to the $d$-partition $\mu$, we have $\dim(\tilde{V}_{\mu}) = \dim(V_{\mu})$. Showing that $V_{\mu}$ is a submodule of $\tilde{V}_{\mu}$ will end the proof of the second part of the theorem.

  Let $\mathfrak{t}$ be a standard Young tableau of shape $\mu$ and denote by $(a_i,b_i,c_i)=\mathfrak{t}^{-1}(i)$.
  Denote by $\mu[i]$ the $d$-partition of $i$ obtained by adding the boxes labeled by $1$ to $i$ in the chosen standard tableau $\mathfrak{t}$ to the empty $d$-partition. We now choose a highest weight vector $v\in M^{\mathfrak{p}}(\Lambda)\otimes V^{\otimes n}$ of weight $\Lambda_\mu$ such that for all $1\leq i \leq n$ we have
  \[
    v\in M^{\mathfrak{p}}(\Lambda,\mu[i])\otimes V^{\otimes (n-i)} \subset M^{\mathfrak{p}}(\Lambda)\otimes V^{\otimes n}.
  \]
 Using the branching rule, one see that such a vector exists and is unique up to a scalar. Let us show that this vector $v$ is a common eigenvector of the Jucys--Murphy elements $X_i$. It is easy to see that the action of the Jucys--Murphy element $X_i$ on $M^{\mathfrak{p}}(\Lambda)\otimes V^{\otimes n}$ is given by the double braiding $(c_{V,M^{\mathfrak{p}}(\Lambda)\otimes V^{\otimes (i-1)}}\circ c_{M^{\mathfrak{p}}(\Lambda)\otimes V^{\otimes (i-1)},V})\otimes \id_{V^{\otimes (n-i)}}$. By \cref{lem:scalar_double}, we obtain that $X_i$ acts on $v$ by multiplication by
  \[
    q^{\langle \Lambda_{\mu[i]},\Lambda_{\mu[i]}+2\rho\rangle - \langle\Lambda_{\mu[i-1]},\Lambda_{\mu[i-1]}+2\rho\rangle - \langle \varepsilon_1,\varepsilon_1+2\rho\rangle}.
  \]
 Indeed, $v$ lies in the summand $M^{\mathfrak{p}}(\Lambda,\mu[i])\otimes V^{\otimes (n-i)} \subset M^{\mathfrak{p}}(\Lambda,\mu[i-1])\otimes V\otimes V^{\otimes (n-i)}$ of $M^{\mathfrak{p}}(\Lambda)\otimes V^{\otimes n}$. But $\Lambda_{\mu[i]} = \Lambda_{\mu[i-1]} + \varepsilon_{k_i}$ where $k_i = \tilde{m}_{c_{i-1}}+a_i$ so that
  \begin{align*}
    \langle \Lambda_{\mu[i]},\Lambda_{\mu[i]}+2\rho\rangle - \langle\Lambda_{\mu[i-1]},\Lambda_{\mu[i-1]}+2\rho\rangle - \langle\varepsilon_1,\varepsilon_1+2\rho\rangle &= 2\langle \Lambda_{\mu[i-1]},\varepsilon_{k_i}\rangle + 2(1-k_i)\\
    &= 2(\beta_{c_i} + b_i-k_i),
  \end{align*}
  since the component of $\Lambda_{\mu[i-1]}$ on $\varepsilon_{k_i}$ is $\beta_{c_i} + (b_i-1)$. Therefore, $X_i$ acts on $v$ by multiplication by
  \[
    (\lambda_{c_i}q^{b_i-k_i})^2=u_{c_i}q^{2(b_i-a_i)}.
  \]
  Therefore, the $\mathcal{H}(d,n)$ submodule spanned by $v$ is isomorphic to $V_{\mu}$ and then $V_{\mu}$ is a submodule of $\tilde{V}_{\mu}$. 
\end{proof}

\subsection{Some particular cases}

We finish by giving some special cases of \cref{thm:main_result} in order to recover various well-known algebras. The two first special cases involve the well-known situation without a parabolic Verma module: it suffices to note that, if $\mathfrak{p}=\mathfrak{gl}_m$, then $M^{\mathfrak{p}}(\Lambda)$ is the trivial module.

\begin{cor}\label{cor:HeckeA}
  If the parabolic subalgebra $\mathfrak{p}$ is $\mathfrak{gl}_m$ and $m\geq n$, then the endomorphism algebra of $M^{\mathfrak{p}}(\Lambda)\otimes V^{\otimes n}$ is isomorphic to Hecke algebra of type $A$.
\end{cor}

\begin{cor}\label{cor:TLtypeA}
  If the parabolic subalgebra $\mathfrak{p}$ is $\mathfrak{gl}_m$ and $m=2$, then the endomorphism algebra of $M^{\mathfrak{p}}(\Lambda)\otimes V^{\otimes n}$ is isomorphic to Temperley--Lieb algebra of type $A$.
\end{cor}

We now turn to special cases where $\mathfrak{p}$ is a strict subalgebra of $\mathfrak{gl}_m$. The following corollary follows from~\cref{rem:asymptotic_AK}.

\begin{cor}\label{cor:AK}
For $\mathfrak{p}$ such that $m\geq nd$ and $m_i\geq n$ for all $1\leq i \leq d$, the endomorphism algebra of $M^{\mathfrak{p}}(\Lambda)\otimes V^{\otimes n}$ is isomorphic to the Ariki--Koike  algebra $\mathcal{H}(d,n)$.
\end{cor}

The Hecke algebra of type $B$ with unequal parameters appears when we work with a standard parabolic subalgebra $\mathfrak{p}$ with Levi factor $\mathfrak{gl}_{m_1}\times\mathfrak{gl}_{m_2}$.

\begin{cor}\label{cor:HeckeB}
  If the parabolic subalgebra $\mathfrak{p}$ is such that $d=2$, $m_1\geq n$ and $m_2\geq n$, then the endomorphism algebra of $M^{\mathfrak{p}}(\Lambda)\otimes V^{\otimes n}$ is isomorphic to the Hecke algebra of type $B$ with unequal and algebraically independent parameters.
\end{cor}

Finally, the last special case is a generalization of the $\mathfrak{gl}_2$ case of \cite{ILZ}, where we recover the generalized blob algebra. 

\begin{cor}\label{cor:genblob}
  If the parabolic subalgebra $\mathfrak{p}$ is the standard Borel subalgebra $\mathfrak{b}$ of $\mathfrak{gl}_m$, that is $d=m$ and $m_i=1$ for $1\leq i \leq d$, then the endomorphism algebra of $M(\Lambda)\otimes V^{\otimes n}$ is isomorphic to the generalized blob algebra $\mathcal{B}(d,n)$.
\end{cor}



%% file: sections/non-semisimple.tex

\section{Some remarks on the non-semisimple case}
\label{sec:non-semisimple}

This paper deals with the semisimple case, where the decomposition of $M^{\frak{p}}(\Lambda)\otimes V^{\otimes n}$ as the sum of simple modules is a crucial tool to compute its endomorphism algebra. Non-semisimple situations appear if $q$ is no longer an indeterminate in the base field $\Bbbk$ but a root of unity. If $q$ and the parameters $\lambda_1,\ldots,\lambda_d$ appearing in the highest weight of $M^{\frak{p}}(\Lambda)$ are no longer algebraically independent, a non-semisimple situation may also appear. Indeed, the parabolic Verma module might not be simple anymore as it is readily seen from the case of $\mathfrak{gl}_2$. It is then natural to ask whether it is possible to extend the Schur--Weyl duality to the non-semisimple case. Let us remark that if $q$ is not a root of unity and if $\lambda_i\lambda_j^{-1}\not\in\mathbb{Z}$ for all $1 \leq i,j \leq d$ then the behavior is similar to the one described in the previous sections.

In order to define the action, we use an ``integral version'' of the algebras $\mathcal{U}_q(\mathfrak{gl}_n)$ and $\mathcal{H}(d,n)$ and of the module $M^{\frak{p}}(\Lambda)\otimes V^{\otimes n}$, compatible with the specialization at a root of unity.

We start with the Ariki--Koike algebra. The definition given in \cref{sec:def_AK} is valid for any field $\Bbbk$ and any choice of parameters. Concerning the algebra $\mathcal{U}_q(\mathfrak{gl}_n)$, we consider Lusztig's integral from $\mathcal{U}_q^{\mathrm{res}}(\mathfrak{gl}_n)$ over $\mathbb{Z}[q,q^{-1}]$, see \cite[Section 9.3]{chari-pressley}. It is also known that the quasi-$R$-matrix $\Theta$ is an element of (a completion of) $\mathcal{U}_q^{\mathrm{res}}(\mathfrak{gl}_n)\otimes \mathcal{U}_q^{\mathrm{res}}(\mathfrak{gl}_n)$. Then for a base field $\Bbbk$ and any $\xi\in \Bbbk^*$, the quantum group $\mathcal{U}_{\xi}(\mathfrak{gl}_n)$ is defined as $\Bbbk\otimes_{\mathbb{Z}[q,q^{-1}]} \mathcal{U}_q^{\mathrm{res}}(\mathfrak{gl}_n)$, where we see $\Bbbk$ as a $\mathbb{Z}[q,q^{-1}]$-module via the morphism sending $q$ to $\xi$.

The parabolic Verma module $M^{\mathfrak{p}}(\Lambda)$ is a highest weight module and we choose $v_\Lambda$ a highest weight vector. We then have at our disposal an integral version, which is the submodule generated over $\mathcal{U}_q^{\mathrm{res}}(\mathfrak{gl}_n)$ by the highest weight $v_\Lambda$. Its specialization at $q=\xi$ will still be denoted $M^{\mathfrak{p}}(\Lambda)$. Similarly, we have a version at $q=\xi$ of the standard module $V$, which has a well-known integral form.

Since the quasi-$R$-matrix $\Theta$ lies in the Lusztig's integral form of the quantum group, we can similarly use the braiding to define the endomorphisms $R_0,R_1,\ldots,R_{n-1}$ of the $\mathcal{U}_{\xi}(\mathfrak{gl}_n)$-module $M^{\mathfrak{p}}(\Lambda)\otimes V^{\otimes n}$. As in the semisimple case, we have:

\begin{prop}
  Let $\Bbbk$ be a field, $q\in \Bbbk^*$ and $\lambda_1,\ldots,\lambda_d \in \Bbbk$. Then the assignment $T_i\mapsto R_i$ is a morphism of algebras from $\mathcal{H}(d,n)$ to $\End_{\mathcal{U}_{\xi}(\mathfrak{gl}_n)}(M^{\mathfrak{p}}(\Lambda)\otimes V^{\otimes n})$. The parameters $u_i$ of the Ariki--Koike algebra are still given by \cref{lem:eignevalues}.
\end{prop}

It is more difficult to understand the image of map $\mathcal{H}(d,n)\rightarrow\End_{\mathcal{U}_{\xi}(\mathfrak{gl}_n)}(M^{\mathfrak{p}}(\Lambda)\otimes V^{\otimes n})$, or even better to describe the image and the kernel of the map.
In~\cite{ILZ} Iohara, Lehrer and Zhang studied the particular case of $\mathfrak{gl}_2$ and $\mathfrak{p}=\mathfrak{b}$ (this corresponds to $n=2$ and $d=2$) and proved that if $q$ is an indeterminate in $\Bbbk$ and that $\lambda_1\lambda_2^{-1}=q^{l}$ for $l\in \mathbb{Z}$, $l\geq -1$, then the map $\mathcal{H}(d,n)\rightarrow\End_{\mathcal{U}_{q}(\mathfrak{gl}_n)}(M^{\mathfrak{p}}(\Lambda)\otimes V^{\otimes n})$ is surjective~\cite[Proposition 5.11]{ILZ}.


In order to extend the Schur--Weyl duality form the semisimple case to a non-semisimple case, a classical strategy \cite{doty,andersen_stroppel_tubbenhauer} is to argue that the dimensions of the various algebras, such as $\End_{\mathcal{U}_{\xi}(\mathfrak{gl}_n)}(M^{\mathfrak{p}}(\Lambda)\otimes V^{\otimes n})$ or $\mathcal{H}(d,n)$, are independent of the base field $\Bbbk$.

Following the arguments of \cite{andersen_stroppel_tubbenhauer}, a first step would be to determine whether the parabolic Verma module $M^{\mathfrak{p}}(\Lambda)$ is tilting in an appropriate category $\mathcal{O}$ of infinite dimensional $\mathcal{U}_q(\mathfrak{gl}_n)$-modules. Since $V$ is tilting and the tensor product of tilting modules is tilting, having $M^{\mathfrak{p}}(\Lambda)$ being tilting would mean that $M^{\mathfrak{p}}(\Lambda)\otimes V^{\otimes n}$ is. Since the space of endomorphisms of a tilting module is flat, its dimension does not depend on the base field $\Bbbk$.

Concerning $\mathcal{H}(d,n)$, its definition is valid over the ring $\mathbb{Z}[q^{\pm 1},u_1,\ldots,u_d]$ and it is known that the basis given in \ref{thm:basis_AK} is a basis over this ring. This implies that the dimension of the algebra $\mathcal{H}(d,n)$ is independent of the field $\Bbbk$ and the choice of $q\in \Bbbk^*$ and of $u_1,\ldots,u_d \in \Bbbk$. 

Therefore, if $M^{\mathfrak{p}}(\Lambda)$ is tilting in an appropriate category $\mathcal{O}$ of infinite dimensional $\mathcal{U}_q(\mathfrak{gl}_n)$-modules, the map $\mathcal{H}(d,n)\rightarrow\End_{\mathcal{U}_{\xi}(\mathfrak{gl}_n)}(M^{\mathfrak{p}}(\Lambda)\otimes V^{\otimes n})$ would be surjective for any base field $\Bbbk$.

If we want to consider the row-quotients $\mathcal{H}_{\underline{m}}(d,n)$ of $\mathcal{H}(d,n)$, one must first give a definition which does not rely on the semisimplicity of the algebra $\mathcal{H}(d,n)$ so that the map $\mathcal{H}(d,n)\rightarrow\End_{\mathcal{U}_{\xi}(\mathfrak{gl}_n)}(M^{\mathfrak{p}}(\Lambda)\otimes V^{\otimes n})$ factors through $\mathcal{H}_{\underline{m}}(d,n)$ and then study the existence of an integral basis of $\mathcal{H}_{\underline{m}}(d,n)$.

Let us stress that these arguments depend heavily on $M^{\mathfrak{p}}(\Lambda)$ being tilting and on the existence of an integral basis of $\mathcal{H}_{\underline{m}}(d,n)$. One may need some extra assumptions on the field $\Bbbk$, as for example being infinite, or on the parameters of the parabolic Verma module. This non-semisimple behavior deserves further study, which was outside the scope of this paper. 


%% file: bibliography/bibliography.tex

\bibliographystyle{bibliography/habbrv}
\bibliography{bibliography/biblio}
